\documentclass[12pt]{amsart}
\usepackage{color}
\usepackage[all,cmtip]{xy}
\usepackage{amssymb}
\usepackage{amsmath}
\usepackage{amsthm}
\usepackage{amscd}
\usepackage{amsfonts}
\usepackage{graphicx}
\usepackage{multicol}
\usepackage{setspace}

\calclayout
\newtheorem{dummy}{anything}[section]
\newtheorem{theorem}[dummy]{Theorem}

\newtheorem{lemma}[dummy]{Lemma}
\newtheorem{proposition}[dummy]{Proposition}
\newtheorem{corollary}[dummy]{Corollary}

\theoremstyle{definition}%%Change Theoremstyle
\newtheorem{definition}[dummy]{Definition}
  \newtheorem{example}[dummy]{Example}
  
  \newtheorem{remark}[dummy]{Remark}

\newcommand{\cAb}{\mathcal Ab}

\newcommand{\cE}{\mathcal E}
\newcommand{\cF}{\mathcal F}

\newcommand{\cH}{\mathcal H}

\newcommand{\bbC}{\mathbb C}
\newcommand{\bbD}{\mathbb D}

\newcommand{\bbR}{\mathbb R}
\newcommand{\bbS}{\mathbb S}
\newcommand{\bbZ}{\mathbb Z}

\DeclareMathOperator{\Hom}{Hom} \DeclareMathOperator{\Aut}{Aut}
 \DeclareMathOperator{\Ext}{Ext}
 
\DeclareMathOperator{\Ind}{Ind} \DeclareMathOperator{\Inn}{Inn}
\DeclareMathOperator{\Res}{Res} 
 \DeclareMathOperator{\Isot}{Isot}
\DeclareMathOperator{\Isog}{Iso}

  \DeclareMathOperator{\Out}{Out}

 \DeclareMathOperator{\rk}{rk}

\newcommand{\iso}{\cong}
\newcommand{\G}{\Gamma}
\newcommand{\U}{\Upsilon}

\newcommand{\Fa}{\cH}  %Family
\newcommand{\Fu}{\cF}   %Fusion

\DeclareMathOperator{\Bdl}{Bdl} \DeclareMathOperator{\Rep}{Rep}
\DeclareMathOperator{\Or}{Or} \DeclareMathOperator{\obs}{obs}
\DeclareMathOperator{\ZOr}{\bbZ O_{\Fa}}

\newcommand{\la}{\langle}
\newcommand{\ra}{\rangle}

\newcommand{\vv}{\, | \,}

\newcommand{\id}{\mathrm{id}}

\newcommand{\diffeo}{\approx}

\newcommand{\leftexp}[2]{{\vphantom{#2}}^{ #1}{\hskip-1pt#2}}%\leftexp
\def\maprt#1{\smash{\,\mathop{\longrightarrow}\limits^{#1}\,}}

\begin{document}

\title{Fusion systems and constructing free actions on
products of spheres}

\author{\" Ozg\" un \" Unl\" u and Erg{\" u}n Yal\c c\i n }

\address{Department of Mathematics, Bilkent
University, Ankara, 06800, Turkey.}

\email{unluo@fen.bilkent.edu.tr \\
yalcine@fen.bilkent.edu.tr}

\keywords{}

\thanks{2000 {\it Mathematics Subject Classification.} Primary:
57S25; Secondary: 20D20.\\ The first author is partially supported
by T\" UB\. ITAK-TBAG/109T384 and the second author is partially
supported by T\" UBA-GEB\. IP/2005-16.}

%\date{\today}

\begin{abstract} We show that every rank two $p$-group acts freely and
smoothly on a product of two spheres.  This follows from a more
general construction: given a smooth action of a finite group $G$ on
a manifold $M$, we construct a smooth free action on $M \times \bbS
^{n_1} \times \dots \times \bbS ^{n_k}$ when the set of isotropy
subgroups of the $G$-action on $M$ can be associated to a fusion
system satisfying certain properties. Another consequence of this
construction is that if $G$ is an (almost) extra-special $p$-group
of rank $r$, then it acts freely and smoothly on a product of $r$
spheres.
\end{abstract}

\maketitle

\section{introduction}
\label{section:Introduction}

In \cite{smith}, P. A. Smith proved that if a finite group $G$ acts
freely on a sphere, then $G$ has no subgroup isomorphic to the
elementary abelian group $\bbZ /p \times \bbZ /p$ for any prime
number $p$. Later in \cite{milnor}, Milnor showed that there are
other restrictions on such a $G$, more precisely, he proved that if
$G$ acts freely on a sphere $\bbS ^n$, then $G$ has no subgroup
isomorphic to the dihedral group $D_{2p}$ of order $2p$ for any odd
prime number $p$.

Conversely, Madsen-Thomas-Wall \cite{madsen-thomas-wall} proved that
the Smith's condition together with Milnor's condition is enough to
ensure the existence of a free smooth action on a sphere $\bbS ^n$
for some $n\geq 1$. The existence proof of Madsen-Thomas-Wall used
surgery theory and exploited some natural constructions of free
group actions on spheres. Specifically, they considered the unit
spheres of linear representations of subgroups to show that certain
surgery obstructions vanish.

As a generalization of the above problem, we are interested in the
problem of characterizing those finite groups which can act freely
and smoothly on a product of two spheres. Similar to Smith's
condition there is a restriction on the existence of such actions:
Heller \cite{heller}, showed that if a finite group $G$ acts freely
on a product of two spheres, then $G$ has no subgroup isomorphic to
the elementary abelian group $\bbZ /p \times \bbZ /p \times \bbZ
/p$. The maximum rank of elementary abelian subgroups $(\bbZ/p)^k
\leq G$ is called the rank of $G$. So, Heller's result says that if
a finite group $G$ acts freely on a product of two spheres, then $G$
must have $\rk (G)\leq 2$. So far, no condition analogous to
Milnor's condition is found for the existence of smooth actions and
it appears as if to prove a converse, we need more constructions of
natural actions.

As a first attempt to construct free actions on products of two
spheres, one can take a product of two unit spheres $\bbS(V_1)\times
\bbS (V_2)$ where $V_1$ and $V_2$ are linear representations of the
group. However, it is not hard to see that for many groups of rank
2, it is not possible to find two linear spheres such that the
action on their product is free. For example, when $p$ is an odd
prime, the extraspecial $p$-group order $p^3$ and exponent $p$ does
not act freely on a product of two linear spheres although this
group has rank equal to two.

Another natural construction is to take a representation with small
fixity and consider the Stiefel manifolds associated to this
representation. The fixity of a $G$-representation $V$ is the
maximum dimension of fixed subspaces $V^g$ over all nontrivial
elements $g$ in $G$. If $G$ has an $n$-dimensional complex
representation of fixity $1$, then $G$ acts freely on the Stiefel
manifold $V_{n,2}(\bbC) \simeq U(n)/U(n-2)$. This space is the total
space of a sphere bundle over a sphere, and taking fiber joins, one
obtains a free action on a product of two spheres. This method was
used by Adem-Davis-\"Unl\"u \cite{adem-davis-unlu} to show that for
$p\geq 5$, every rank two $p$-group acts freely and smoothly on a
product of two spheres. However, there are examples of rank two
$2$-groups and $3$-groups which have no representation with fixity
1, so this method is not enough to construct free actions of rank
two $p$-groups on products of two spheres for all primes $p$.

A more general idea for constructing free actions on a product of
two spheres is to start with a representation sphere $\bbS(V)$ and
construct a $G$-equivariant sphere bundle over it so that the action
on the total space is free and the bundle is non-equivariantly
trivial. In the homotopy category, a similar idea was used by
Adem-Smith \cite{adem-smith} to show that many rank two finite
groups can act freely on a finite complex homotopy equivalent to a
product of two spheres. In particular, they showed that every rank
two $p$-group acts freely on a finite CW-complex homotopy equivalent
to a product of two spheres.

In this paper, we prove the following:

\begin{theorem}\label{theorem:RankTwoPGroups}
A finite $p$-group $G$ acts freely and smoothly on a product of two
spheres if and only if\ $\rk (G)\leq 2$.
\end{theorem}

The proof uses another method of construction of free actions on
products of spheres which was introduced by \" Unl\"u in his thesis
\cite{unlu}. The method uses a theorem of L\" uck-Oliver
\cite[Thm.~2.6]{luck-oliver} on constructions of equivariant vector
bundles over a finite dimensional $G$-CW-complex. We now describe
briefly the main idea of the L\" uck-Oliver construction: Let $X$ be
a finite dimensional $G$-CW-complex and $\Fa$ be the family of
isotropy subgroups of $X$. Given a compatible family of unitary
representations $\rho _H :H \to U(n)$ where $H \in \Fa$, one would
like to construct a $G$-equivariant vector bundle over $X$ so that
the representation over a point with isotropy $H$ is isomorphic to
$(\rho _H) ^{\oplus k}$ for some $k$. L\" uck-Oliver
\cite{luck-oliver} shows that this can be done if there is a finite
group $\Gamma$ which satisfies the following two conditions:
\begin{enumerate} \item There is a family of maps $\{\alpha _H : H
\to \G \vv H \in \Fa\}$ which is compatible in the sense that if
$c_g: H\to K$ is a map induced by conjugation with $g\in G$, then
there is a $\gamma \in \Gamma$ such that the following diagram
commutes:
$$\xymatrix{
H \ar[d]^{c_g} \ar[r]^{\alpha _H}
& \Gamma \ar[d]^{ c_{\gamma }}  \\
K  \ar[r]^{\alpha _K}
&  \Gamma  }\\
$$
\item $\Gamma$ has a representation $\rho : \Gamma \to U(n)$ such
that $\rho _H =\rho \circ \alpha _H$ for all $H \in \Fa$.
\end{enumerate}

In \cite{unlu}, \" Unl\" u showed that when all the groups in the
family $\Fa$ are cyclic $p$-groups, there is a finite group $\Gamma$
satisfying the above conditions for a family of representations
$\rho_H$ such that $H$ action on the unit sphere $\bbS(\rho_H)$ is
free. As a result of this, \" Unl\" u was able to show that when $p$
is odd, every rank two $p$-group acts freely and smoothly on a
product of two spheres. When $p=2$, the groups one has to deal with
are rank one $2$-groups and these can be cyclic or generalized
quaternion. It turns out that maps between subgroups of quaternion
groups have a much richer structure, so the method given in
\cite{unlu} does not extend directly to families of rank one
$2$-groups.

In this paper, we find a systemic way of constructing a finite group
$\Gamma$ satisfying the above conditions (i) and (ii) for some
suitable representation families. We first choose a finite group $S$
and map all subgroups in the family $\Fa$ into $S$ via some maps
$\iota _H : H \to S$. Then, we study the fusion system on $S$ that
comes from the conjugations in $G$ and different choices of mappings
$\iota_H$ (see \cite{broto-levi-oliver} or \cite{ragnarsson-stancu}
for a definition of a fusion system). Although the fusion systems
that arises in this way are not necessarily saturated, we use some
of the machinery developed for studying saturated fusion systems. In
particular, we use a theorem of S. Park \cite{park} to find $\Gamma$
as the automorphism group of an $S$-$S$-biset.

This method of finding a finite group $\Gamma$ works for more
general groups then the families formed by rank one $2$-groups. For
example, for a family $\Fa$ formed by elementary abelian $p$-groups,
we can easily find a finite group $\Gamma$ by choosing an
appropriate $S$-$S$-biset. Moreover, this process can be recursively
continued to obtain the following theorem.

\begin{theorem}\label{theorem:ElementaryAbelianIsotropy}
Let $G$ be a finite group acting smoothly on a manifold $M$. If all
the isotropy subgroups of $M$ are elementary abelian groups with
rank $\leq k$, then $G$ acts freely and smoothly on $M \times
\bbS^{n_1} \times \dots \times \bbS^{n_k}$ for some positive
integers $n_1,\dots, n_k$.
\end{theorem}

As a corollary, we obtain the following:

\begin{corollary}\label{corollary:ExtraSpecial}
Let $G$ be an (almost) extraspecial $p$-group of rank $r$. Then, $G$
acts freely and smoothly on a product of $r$ spheres.
\end{corollary}

The paper is organized as follows: Sections
\ref{section:EquivariantPrincipalBundles} and
\ref{section:EquivariantObstructionTheory} are preliminary sections
on equivariant principal bundles and equivariant obstruction theory.
In Section \ref{section:ConstructionofEquivariantBundles}, we review
the work of L\" uck-Oliver \cite{luck-oliver} on constructions of
equivariant bundles and prove Theorem
\ref{theorem:LuckOliverconstruction} which is a slightly different
version of Theorem 2.7 in \cite{luck-oliver}. Then, in Section
\ref{section:ConstructionOfFiniteGroupGamma}, we introduce a method
for constructing finite groups $\Gamma$ satisfying the properties
explained above. This is done using a theorem of S. Park \cite{park}
on bisets associated to fusion systems. Finally, in Section
\ref{section:ConstructionsofFreeActions}, we prove our main
theorems, Theorem \ref{theorem:RankTwoPGroups} and
\ref{theorem:ElementaryAbelianIsotropy}.

\section{Equivariant Principal Bundles}
\label{section:EquivariantPrincipalBundles}

In this section, we introduce the basic definitions of equivariant
bundle theory. We refer the reader to \cite{luck1}, \cite{luck2},
and \cite{luck-oliver} for more details.

Let $G$ be a compact Lie group. A relative $G$-CW-complex $(X,A)$ is
a pair of $G$-spaces together with a $G$-invariant filtration
$$A = X^{(-1)} \subseteq X^{(0)} \subseteq X^{(1)} \subseteq \ldots \subseteq
X^{(n)} \subseteq \ldots \subseteq \bigcup_{n \geq -1} X^{(n)} = X$$
such that $A$ is a Hausdorff space, $X$ carries the colimit topology
with respect to this filtration and for all $n \geq 0$,  the space
$X^{(n)}$ is obtained from $X^{(n-1)}$ by attaching equivariant
$n$-dimensional cells, i.e., there exists a $G$-pushout diagram as
follows
$$\begin{CD}
\underset{\sigma \in I_n}{\coprod } G/H_{\sigma } \times \bbS ^{n-1} @>{}>> X^{(n-1)}  \\
@V{ }VV  @VV{ }V\\
\underset{\sigma \in I_n}{\coprod } G/H_{\sigma } \times \bbD ^{n}
@>{}>> X^{(n)}
\end{CD}$$
where  $I_n$ is an index set, $H_{\sigma }$ is a subgroup of $G$ for
$\sigma \in I_n$, and $n\geq 0$. Elements of $I_n$ are called
equivariant $n$-cells and for $\sigma \in I_n$, the map $
G/H_{\sigma } \times \bbS ^{n-1} \to X^{(n-1)}$ is called the
attaching map and the map $G/H_{\sigma } \times \bbD ^{n}\to
X^{(n)}$ is called the characteristics map of the cell. Here we
consider $\bbS ^{-1}=\emptyset $ and $\bbD ^{0}=\{\text{a point}\}$.
The space $X^{(n)}$ is called the $n$-skeleton of $(X,A)$ for $n\geq
-1$. For more details about $G$-CW-complexes, see Section II.1-2 in
\cite{dieck} and Section I.1-2 in \cite{luck1}.

We now give the definition for the classifying space of a group
relative to a family.

\begin{definition}
\label{definition:EFG} Let $\Fa $ be a family of closed subgroups of
$G$ closed under conjugation. Define $E_{\Fa }(G)$ as the
realization of the nerve of the category $\cE _{\Fa } (G)$ whose
objects are pairs $(G/H,xH)$ where $H\in \Fa $ and $x\in G$ and
morphisms from $(G/H,xH)$ to $(G/K,yK)$ are the $G$-maps from $G/H$
to $G/K$ which sends $xH$ to $yK$.
\end{definition}

We can consider the space $E_{\Fa }(G)$ as a $G$-CW-complex with the
$G$-action induced by $g(G/H, xH)=(G/H, gxH)$ on the objects of the
category $\cE _{\Fa } (G)$. For any $H\in \Fa$, the space $E_{\Fa
}(G)^H$ is the realization of the nerve of the full subcategory of
$\cE _{\Fa } (G)$ with objects $(G/K,xK)$ where $H\leq K^x$. The
object $(G/H,H)$ is an initial object in this subcategory. Hence
$E_{\Fa }(G)^H$ is contractible for any $H\in \Fa$ and we get the
following classifying property of $E_{\Fa }(G)$.

\begin{proposition}\cite[Prop 2.3]{luck1}
Let $(X, A)$ be a relative $G$-CW-complex such that $G_x\in \Fa $
for all $x\in X$. Then, any $G$-map from $A$ to $E_{\Fa }(G)$
extends to a $G$-map from $X$ to $E_{\Fa }(G)$ and any two such
extensions are $G$-homotopic relative to $A$.
\end{proposition}

Let $G$ be a finite group and $\Gamma $ be a compact Lie group. A
$G$-equivariant $\Gamma $-bundle over a left $G$-space $X$ is a
$\Gamma $-principal bundle $p:E\to X$ where $E$ is a left $G$-space,
$p$ is  a $G$-equivariant map, and the right action of $\Gamma $ on
$E$ and the left action of $G$ on $E$ commute. Let $\Bdl_{G,\Gamma
}(X)$ denote the isomorphism classes of $G$-equivariant $\Gamma
$-bundles over $X$.

Let $\Or _{\Fa}(G)$ denote the orbit category whose objects are
orbits $G/H$ where $H\in \Fa $ and morphisms are $G$-maps from $G/H$
to $G/K$. Assume that we are given an element
$$ {\mathbf A}=(p_H)\in \lim _{\underset{(G/H)\in \Or _{\Fa}(G)}{\longleftarrow
}} \Bdl_{G,\Gamma }(G/H)\subseteq \underset{H\in \Fa}{\prod
}\Bdl_{G,\Gamma }(G/H) $$ where a $G$-map from $G/H$ to $G/K$
induces a function from $\Bdl_{G,\Gamma }(G/K)$ to $\Bdl_{G,\Gamma
}(G/H)$ by pullbacks. A $G$-equivariant ${\mathbf A}$-bundle over a
left $G$-space $X$ is a $G$-equivariant $\Gamma $-bundle $p:E\to X$
such that for any $H\in \Fa $ and any $G$-equivariant map $i:G/H\to
X$, the pullback $i^*(p)$ is isomorphic to $p_H$ in $\Bdl_{G,\Gamma
}(G/H)$. Let $\Bdl_{G,{\mathbf A} }(X)$ denote the isomorphism
classes of $G$-equivariant ${\mathbf A}$-bundles over $X$.

\begin{lemma} We have
$$\Bdl_{G,\Gamma }(G/H)\iso \Rep (H,\Gamma ):=\Hom (H,\Gamma )/\Inn (\Gamma )$$
where $\Hom (H,\Gamma )$ is the set of homomorphisms from $H$ to
$\Gamma $ and $\Inn (\Gamma )$ is the group of inner automorphisms
of $\Gamma $ and the action of $\Inn (\Gamma )$ on $\Hom (H,\Gamma
)$ is given by composition.
\end{lemma}

\begin{proof} For a $G$-equivariant $\Gamma $-bundle $p_H$ over the $G$-space
$G/H$, let $E(p_H)$ denote the total space of the bundle $p_H$. Take
a point $x\in p_H^{-1}(H)\subseteq E(p_H)$. Since $G\times \Gamma$
acts transitively on $E(p_H)$, we have $E(p_H)=(G\times
\Gamma)/(G\times \Gamma)_x$ and $(G \times \Gamma)_x \cap (1 \times
\Gamma)=\{ 1\}$. So, by Goursat's lemma,
$$(G\times \Gamma )_x=\Delta ( \alpha_x):=\left\{(h,\alpha_x(h))
\vv h\in H\right\} $$ where the homomorphism $\alpha _x:H\to \Gamma$
is defined by the equation $hx(\alpha_x(h))^{-1}=x$ for $h\in H$.
Let $f$ be a bundle isomorphism from $p_H$ to another
$G$-equivariant $\Gamma $-bundle $q_H$ over the $G$-space $G/H$.
Take $y\in q_H^{-1}(H)\subseteq E(q_H)$ and define $\alpha _{y}:H\to
\Gamma$ as above. Then there exists $\gamma _{x,y}\in \Gamma $ such
that $f(x)=y\gamma_{x,y}$. So, for all $h\in H$, we have $\alpha
_y(h)=\gamma_{x,y}\alpha _{x}(h)\gamma_{x,y}^{-1}$. Hence, up to
composition with an inner automorphism of $\Gamma $, there exists a
unique map $\alpha _H:H\to \Gamma $ such that $E(p_H)\iso G\times _H
\Gamma $ where the action of $H$ on $G\times \Gamma $ is given by
$h(g,\gamma )=(gh^{-1},\alpha_H(h)\gamma )$.
\end{proof}

We can view the family $\Fa $ as a category where the elements of
$\Fa $ are the objects of the category and morphisms are
compositions of conjugations in $G$ with inclusions. A morphism in
the category $\Or _{\Fa}(G)$ is a $G$-map from $G/H$ to $G/K$ and
can be written in the form $\hat{a}:G/H \to G/K$ where
$\hat{a}(gH)=ga^{-1}K$ for $a\in G$ such that $aHa^{-1}\leq K$. Now
the map induced by $\hat{a}$ from $\Bdl_{G,\Gamma }(G/K)$ to
$\Bdl_{G,\Gamma }(G/H)$ by pullbacks is equivalent to the map from
$\Rep(K,\Gamma )$ to $\Rep(H,\Gamma )$ induced by conjugation
$c_a:H\to K$ given by $c_a(h)=aha^{-1}$. Hence we can consider
%${\mathbf A} \in \lim \Bdl_{G,\Gamma }(G/H) $
$$ {\mathbf A}=(p_H)\in \lim _{\underset{(G/H)\in \Or _{\Fa}(G)}{\longleftarrow
}} \Bdl_{G,\Gamma }(G/H)$$ as an element
$${\mathbf A}=(\alpha _H)\in \lim _{\underset{H\in \Fa
}{\longleftarrow }} \Rep (H,\Gamma )\subseteq \underset{H\in
\Fa}{\prod }\Rep (H,\Gamma ).$$ A family of representations $\alpha
_H: H \to \Gamma$ is called a {\it compatible family} of
representations if it is an element of a limit as above. We now
describe the classifying space for $G$-equivariant ${\mathbf
A}$-bundles.

\begin{definition} Let ${\mathbf A}=(\alpha _H)$ be as above and let
$\Fa _{\mathbf A}$ be the family of subgroups $W \leq G\times
\Gamma$ such that $W=\text{graph} (\alpha _H)$ for some
representation $\alpha _H$ in ${\mathbf A}$. Define
$$E_{\Fa}(G,{\mathbf A})=E_{\Fa _{\mathbf A}}(G\times \Gamma)\text{\ \
and \ \ }B_{\Fa}(G,{\mathbf A})=E_{\Fa _{\mathbf A}}(G,{\mathbf
A})/\{1\}\times \Gamma .$$
\end{definition}

Note that the $G$-equivariant $\Gamma$-principal bundle
$E_{\Fa}(G,{\mathbf A})\to B_{\Fa}(G,{\mathbf A})$ is indeed a
$G$-equivariant ${\mathbf A}$-bundle. This is because for any $x\in
E_{\Fa}(G,{\mathbf A})$, we have $x\Gamma \in B_{\Fa}(G,{\mathbf
A})$ and the pullback of the bundle $E_{\Fa}(G,{\mathbf A})\to
B_{\Fa}(G,{\mathbf A})$ by the natural inclusion of $G/G_{x\Gamma
}\to B_{\Fa}(G,{\mathbf A})$ is isomorphic  to the bundle $(G\times
\Gamma)/(G\times \Gamma)_x\to G/G_{x\Gamma }$. We also know that
$(G\times \Gamma)_x\in \Fa _{\mathbf A}$ hence $(G \times \Gamma
)_x= \Delta (\alpha _H )$ for some $\alpha _H$ in ${\mathbf A}$. In
particular, $H=G_{x\Gamma}$. Hence the pullback of the bundle
$E_{\Fa}(G,{\mathbf A})\to B_{\Fa}(G,{\mathbf A})$ by the natural
inclusion of $G/G_{x\Gamma }\to B_{\Fa}(G,{\mathbf A})$ is
isomorphic to $p_{G_{x\Gamma }}$ in ${\mathbf A}$. We have the
following:

\begin{proposition}\cite[Lemma 2.4]{luck-oliver} Let
$\displaystyle {\mathbf A}=(\alpha _H : H \to \Gamma)$ be a
compatible family of representations. Then, the following hold:
\begin{enumerate}
\item The bundle $E_{\Fa} (G, {\mathbf A}) \to B_{\Fa} (G,
{\mathbf A})$ is the universal $G$-equivariant  ${\mathbf
A}$-bundle: If $X$ is a $G$-CW-complex such that $G_x\in \Fa $ for
all $x\in X$, then the map defined by pullbacks
$[X,B_{\Fa}(G,{\mathbf A})]_G \to \Bdl _{G, {\mathbf A}}(X)$ is a
bijection.
\item For all $H \in \Fa$, we have $B_{\Fa}(G,{\mathbf A}) ^H \simeq
BC_{\Gamma }(\alpha _H)$ where $C_{\G} (\alpha _H)$ denotes the
centralizer of the image of $\alpha_H$ in $\G$.
\end{enumerate}
\end{proposition}

\begin{proof} For the first statement observe that if $E\to X$ is a
$G$-equivariant ${\mathbf A}$-bundle, then by construction there is
a $(G\times \Gamma)$-map from $E$ to $E_{\Fa} (G, {\mathbf A})$.
Since both spaces have free $\Gamma$-action, taking orbit spaces we
get a $G$-map $X \to B_{\Fa} (G, {\mathbf A})$ where the bundle
$E\to X$ is the pullback bundle via this map.

To prove the second statement, let $\alpha_H :H\to \Gamma $ be a
representation and let $C=\{1\}\times C_{\Gamma }(\alpha_H)$. Then
$C$ acts freely on the contractible space $E_{\Fa}(G,{\mathbf
A})^{\Delta (\alpha _H)}$ and
$$E_{\Fa}(G,{\mathbf A})^{\Delta(\alpha_H)}/C\iso
B_{\Fa}(G,{\mathbf A})^H$$ where the homeomorphism is given by
$f(xC)=x\Gamma $ for $x\in E_{\Fa }(G,{\mathbf A})^{\Delta (\alpha
_H)}$. To see that $f$ is an isomorphism, first note that $f$ is
well-defined and the image of $f$ is in $B_{\Fa}(G,{\mathbf A})^H$.
Now, take $x,y \in E_{\Fa} (G, {\mathbf A})$ such that
$f(xC)=f(yC)$. Then, $x=y\gamma$ for some $\gamma \in \Gamma$. Since
$hx=x\alpha_H (h)$ and $hy=y \alpha_H (h)$ for all $h \in H$, we get
$\alpha_H (h)\gamma=\gamma \alpha_H(h)$ for all $h \in H$. Thus
$\gamma \in C_{\G} (\alpha _H)$ and $xC=yC$. This proves that $f$ is
injective. To show that $f$ is surjective, let $x\Gamma \in
B_{\Fa}(G,{\mathbf A})^H $. Then $H\leq G_{x\Gamma }$ and there
exists $\beta : G_{x\Gamma }\to \Gamma $ in ${\mathbf A}$ such that
$(G\times \Gamma )_x=\Delta( \beta)$. Since the family of maps in
${\mathbf A}$ are compatible, there is a $\gamma \in \Gamma $ such
that $c_{\gamma}\circ \beta | _H=\alpha_H $. Then $\Delta (\alpha
_H) \leq (G\times \Gamma )_{x\gamma }$. This means that $x\gamma \in
E_{\Fa} (G, {\mathbf A})^{\Delta (\alpha _H)}$ and a direct
calculation gives $f(x\gamma C)=x\Gamma $. So, $f$ is surjective.
Hence we conclude that $B_{\Fa}(G,{\mathbf A})^H$ is homotopy
equivalent to $BC_{\Gamma} (\alpha _H)$.
\end{proof}

\section{Equivariant Obstruction theory}
\label{section:EquivariantObstructionTheory}

In this section, we fix our notation for Bredon cohomology and state
the main theorem of the equivariant obstruction theory that will be
used in the next section. We refer the reader to \cite{bredon},
\cite{luck1}, and \cite{luck-oliver} for more details.

Let $G$ be a finite group and $\Fa$ be a family of subgroup closed
under conjugation. As before we denote the orbit category of $G$
relative to the family $\Fa$ by $\Or _{\Fa} (G)$. Let $(X,A)$ be a
relative $G$-CW-complex whose all isotropy groups are in $\Fa $. A
coefficient system for Bredon cohomology is a contravariant functor
$M:\Or _{\Fa }(G)\to \cAb $ where  $\cAb $ denotes the category of
abelian groups and group homomorphisms between them. A coefficient
system is sometimes called a $\bbZ\Or _{\Fa} (G)$-module with the
usual convention of modules over a small category. So, morphisms
between $\bbZ\Or _{\Fa} (G)$-modules are given by a natural
transformation of functors. Notice that the $\bbZ\Or _{\Fa}
(G)$-module category is an abelian category, so the usual
constructions of modules over a ring are available to do homological
algebra. To simplify the notation, we call a $\bbZ\Or _{\Fa}
(G)$-module, a $\ZOr$-module.

Now, let us fix some notation for some of the $\ZOr$-modules that we
will be considering. For example, consider the contravariant functor
$\pi _n(X^{?},A^{?}):\Or _{\Fa }(G)\to \cAb$ given by $\pi
_n(X^{?},A^{?})(G/H)=\pi _n(X^{H},A^{H})$ for any object $G/H$ in
$\Or _{\Fa }(G)$ and a morphism $\hat{a}:G/H\to G/K$ in $\Or _{\Fa
}(G)$ defined by $gH\to ga^{-1}K$ is sent to the morphism from $\pi
_n(X^{K},A^{K})$ to $\pi _n(X^{H},A^{H})$ induced by left
multiplication $x\to a^{-1}x$ considered as a map from
$(X^{K},A^{K})$ to $(X^{H},A^{H})$.

Similarly, we set $C_n(X^{?},A^{?};M):\Or _{\Fa }(G)\to \cAb $ as
$$C_n(X^{?},A^{?})(G/H)=C_n(X^H,A^H;\bbZ )$$ and morphisms defined
in a similar way as above. The boundary maps of the chain complexes
$C_*(X^H,A^H;\bbZ )$ commute with conjugation and restriction maps,
so when we put them together, we obtain a chain complex of
$\ZOr$-modules
$$\begin{CD}
\cdots @>>> C_2(X^{?},A^{?}) @>{\partial _1}>> C_1(X^{?},A^{?})
@>{\partial _0}>> C_0(X^{?},A^{?}) @>>> 0.
\end{CD}$$
We define $H_n (X^{?}, A^{?})$ as the cohomology of this chain
complex. Note that the $\ZOr$-module $H_n(X^{?},A^{?};M):\Or _{\Fa
}(G)\to \cAb $ satisfies
$$H_n(X^{?},A^{?})(G/H)=H_n(X^H,A^H;\bbZ ).$$

\begin{definition} Let $(X,A)$ be a relative $G$-CW-complex and $M$
be a $\ZOr$-module. The Bredon cohomology $H_G^*(X,A;M)$ of the pair
$(X,A)$ with coefficients in $M$ is defined as the cohomology of the
cochain complex
$$\begin{CD}
0 @>>> \Hom  _{\ZOr}(C_0(X^{?},A^{?}),M) @>{\delta ^0}>> \Hom _
{\ZOr}(C_1(X^{?},A^{?}),M) @>{\delta ^1 }>>\cdots
\end{CD}$$
\end{definition}

Bredon cohomology is useful to describe obstructions for extending
equivariant maps. Let $(X,A)$ be a relative $G$-CW-complex and $Y$
be a $G$-space such that for all $H\leq G$ the invariant space $Y^H$
is an $(n-1)$-simple space. Assume $f:X^{(n)}\to Y$ is a
$G$-equivariant map. Then we define an element $c_f$ in $\Hom
_{\ZOr}(C_n(X^{?},A^{?}),\pi _{n-1}(Y^{?}))$ for $H\in \Fa $ as
follows: For every $H\in \Fa$, the homomorphism $c_f (H)$ is the map
$$c_f(H): C_n (X^H, A^H) \to \pi _{n-1} (Y^H)$$ which takes $\sigma
\in C_n(X^H, A^H)$ to the homotopy class of the map $f\circ \phi
_{\sigma }:\bbS ^{n-1}\to Y^H$ where $\phi _{\sigma }$ is the
attaching map of the cell $\sigma $ in the following pushout
diagram:
$$\begin{CD}
\partial (\sigma ) @>{\phi _{\sigma }}>> X^{(n-1)}  \\
@V{ }VV  @VV{ }V\\
\sigma @>{}>> X^{(n)}
\end{CD}$$
The cochain $c_f$ is a cocyle by Proposition II.1.1 in
\cite{bredon}. Hence we can define $\obs(f)=[c_f]\in H_G^n(X,A; \pi
_{n-1}(Y^{?}))$. The cohomology class $\obs(f)$ is the obstruction
to extending $f|_{X^{(n-1)}}$ to $X^{(n+1)}$. More precisely:

\begin{proposition}\label{proposition:equivariant_obstruction}
Let $(X,A)$  be a relative $G$-CW-complex and $Y$
 be a $G$-space such that for all $H\leq G$, the invariant
space $Y^H$ is a simple space. Let $f:X^{(n)}\to Y$ be a
$G$-equivariant map. Then $f|_{X^{(n-1)}}$ can be extended to an
equivariant map from $X^{(n+1)}$ to $Y$ if and only if $\obs(f)=0$
in  $H_G^{n+1}(X,A; \pi _{n}(Y^{?}))$.
\end{proposition}

\begin{proof} See Proposition II.1.2 in \cite{bredon}.
\end{proof}

Note that the category of $\ZOr$-modules has enough injectives (see
\cite[pg. 24]{bredon}). Hence for any $\ZOr$-module $M$, there
exists an injective resolution
$$\begin{CD}
0 @>>> M @>{\epsilon }>> I^0   @>{\rho ^0}>> I^1 @>{\rho ^1
}>>\cdots
\end{CD}$$
For a $\ZOr$-module $N$, we define the ext-group $\Ext ^n_{\ZOr}(N,
M )$ as the cohomology of the cochain complex
$$\begin{CD}
0 @>>> \Hom _{\ZOr} (N, I^0)   @>{(\rho ^0)_*}>> \Hom _{\ZOr} (N,
I^1 ) @>{(\rho ^1) _*}>>\cdots
\end{CD}$$

Note that since we already know that the $\ZOr$-module category has
enough projectives, one can also calculate the above ext-groups
using a projective resolutions of $N$.

The following  proposition is used in the next section. We include a
proof of it here for the convenience of the reader. The proof is
given by standard homological algebra and can be found in the
literature (see \cite[Chp1, 10.4]{bredon} or \cite[Chp. 1, Thm
6.2]{may}).

\begin{proposition}\label{proposition:spectral_sequence}
Let $(X,A)$ be a $G$-CW-complex and $\Fa $ be a family of subgroups
of $G$ closed under conjugation such that for all $x \in X$, the
isotropy subgroup $G_x$ is in the family $\Fa $. Then there exists a
spectral sequence
$$E_2^{p,q}=\Ext ^p_{\ZOr}(H_q(X^{?},A^{?}), M ) )
\implies  H_G^{p+q}(X,A; M)$$ where $M$ is an $\ZOr$-module.
\end{proposition}

\begin{proof} Let $(C_* (X^?, A^?), \partial)$ denote the chain complex of $(X,A)$
and let
$$\begin{CD}0 @>>> M @>{\epsilon }>> I^0   @>{\rho ^0}>> I^1 @>{\rho ^1
}>>I^2 @>{\rho ^2 }>>\cdots
\end{CD}$$
be an injective resolution of $M$ as a $\ZOr$-module. Define a
double complex
$$D^{p,q}=\Hom_{\ZOr}(C_q(X^?,A^?),I^p) $$
where $d_1: D^{p,q}\to D^{p+1,q}$ is given by $d_1(f)=\rho ^{p}\circ
f$ and $d_2: D^{p,q}\to D^{p,q+1}$ is given by $d_2(f)=(-1)^p\,
f\circ \partial _{q+1}$  for $f\in D^{p,q}$. Now the spectral
sequence of this double complex is in the form
$$E_2^{p,q}=H^p \left( H^q \left( D^{*,*}, d_2\right),d_1 \right)
\implies  H^{p+q}({\rm Tot}(D^{*,*}), d_1+d_2)$$ where ${\rm
Tot}(D^{*,*})$ is the total complex of the double complex $D^{*,*}$
(see page 108 in \cite{bensonII}).

Since $I^p$ is injective for all $p\geq 0$, we have
$$H^q \left( D^{p,*}, d_2\right)=
H^q \left(\Hom_{\ZOr}(C_*(X^?,A^?),I^p) , d_2\right)
=\Hom_{\ZOr}(H_q(X^?,A^?),I^p).$$ Using this and the definition of
ext-groups, we obtain
$$E_2^{p,q}= H^p \left(\Hom_{\ZOr}(H_q(X^?,A^?),I^*),d_1 \right)
=\Ext ^p_{\ZOr}(H_q(X^?,A^?), M ) ).$$

Since $C_q(X^?,A^?)$ is projective as a $\ZOr$-module for all $q\geq
0$, the following two cochain complexes are chain homotopy
equivalent
$$({\rm Tot} (D^{*,*}), d_1+d_2)\simeq (\Hom_{\ZOr}(C_*(X^?,A^?),M),d_2) $$
(see page 45 in \cite{bensonI}). Hence $$ H^{p+q}({\rm
Tot}(D^{*,*}), d_1+d_2)=H_G^{p+q}(X,A; M).$$ Therefore the spectral
sequence for the double complex $D^{*,*}$ gives a spectral sequence
$$E_2^{p,q}=\Ext ^p_{\ZOr}(H_q(X^{?},A^{?}), M ) )
\implies  H_G^{p+q}(X,A; M).$$
\end{proof}

\section{Construction of Equivariant Bundles}
\label{section:ConstructionofEquivariantBundles}

The main theorem of this section is a slightly different version of
a theorem of L\" uck and Oliver \cite[Thm 2.7]{luck-oliver} on
construction of equivariant bundles. This is the theorem that was
mentioned in the introduction and it is the starting point of our
construction of free actions on products of spheres.

Let $\U _k$ be a family of topological groups indexed by positive
integers. Given two maps $f,g: \U_k \to \U_m$, the multiplication of
their homotopy classes $[f]$ and $[g]$ in $[\U _k, \U _l]$ is
defined as the homotopy class of the composition
$$\U _k\overset{\Delta }{\longrightarrow }
\U _k\times \U _k\overset{f\times g}{\longrightarrow }
 \U _m\times \U _m \overset{\mu }{\longrightarrow }\U _m $$
where $\Delta $ denotes the diagonal map and $\mu $ is the
multiplication in $\U _m$.

Let $i_k$ and $j _k$ be injective homomorphisms from $\U _k$ to $\U
_{k+1}$. For $m>k$, let $$i_{k,m}, j_{k,m}: \U_k \to \U_{m}$$ denote
the compositions $i_{m-1}\circ i_{m-2}\circ \ldots \circ i_k $ and
$j_{m-1}\circ j_{m-2}\circ \ldots \circ j_k $ respectively.

\begin{definition}\label{definition:powering tower}
We call a sequence of triples $\left\{(\U _k ,i_k,
j_k)\right\}_{k=1}^{\infty }$ an $r$-powering tower if for each
$k\geq 1 $,  the centralizer of every finite subgroup of $\U _k$ is
a path connected group, and for all $m>k$, we have
$$[i_{k,m}]\simeq
\underset{r^{(m-k)}-\text{many}}{\underbrace{[j_{k,m}]\cdot[j_{k,m}]\cdot
 \ldots \cdot [j_{k,m}]}}.$$
\end{definition}

The main example of a powering tower is the following:

\begin{example}\label{example:Un_powering_tower}
For $k\geq 1$, let $\U_k=U(nr^{k-1})$ and $i_k$ and $j_k$ be the
inclusions from $U(nr^{k-1})$ to $U(nr^k)$ given by $$i_k(A)=\left[
\begin{matrix}
    A &   &  & \\
      & A &  & \\
      &   & \ddots & \\
      &   &  & A \\
\end{matrix} \right]
\text{\ \ \
and \ \ \ }
j_k(A)=\left[
\begin{matrix}
    A &   &  & \\
      & I &  & \\
      &   & \ddots & \\
      &   &  & I \\
\end{matrix} \right].
$$
The centralizer of a finite group in $\U _k=U(nr^{k-1})$ for $k\geq
1$ is isomorphic to a product $\prod _i U(m_i)$ of unitary groups,
hence it is path connected. Let $H_s:[0,1]\to U(nr^k)$ be a path
with the following end points:
$$ H_s(0)=\left[\begin{matrix}
    A &   &  & \\
      & I &  & \\
      &   & \ddots & \\
      &   &  & I \\
\end{matrix} \right]
\text{\ \ \
and \ \ \ }
H_s(1)=\left[
\begin{matrix}
    I &   &        &     &    &  &   \\
      & \ddots &        &     &    &  &       \\
      &   & I &     &    &     & \\
      &   &        &  A  &    &   &   \\
      &   &        &     & I  &   & \\
       &   &        &     &   &  \ddots &  \\
       & & & &  & & I
\end{matrix} \right] \longleftarrow s^{\text{th}}\text{\ position }
$$
Now $\prod_{s=1}^{r}{H_s}$ is a path from $j_k(A)^r$ to $i_k(A)$, so
we get $i_k \simeq (j_k)^r$ for all $k\geq 1$. Applying this
recursively, we obtain  $i_{k,m}\simeq (j_{k,m})^{r^{m-k}}$ for
every $m>k$. Hence $(\U_k, i_k, j_k)$ is an $r$-powering tower.\qed
\end{example}

In our applications, the only $r$-powering tower we consider is the
tower given in the above example. So, one can read the rest of this
section with this particular tower in mind. The reason we keep the
exposition more general is that we believe this more general set up
can be useful for constructing equivariant fibre bundles with fibres
homeomorphic to a product of spheres.

Now we give our main construction.

\begin{theorem}[Compare to Theorem 2.7 in \cite{luck-oliver}]
\label{theorem:LuckOliverconstruction} Let $G$ be a finite group and
$\Fa $ be a family of subgroups of $G$ closed under conjugation.
Suppose that $\Gamma $ is a finite group and
$$ {\mathbf A}=(\alpha _H)\in \lim _{\underset{H\in \Fa }{\longleftarrow
}} \Rep (H,\Gamma ).$$ Let $\left\{(\U _k,i_k,
j_k)\right\}_{k=1}^{\infty }$ be a $|\Gamma |$-powering tower. Then,
for any representation $\rho: \Gamma \to \U _1$ and for any $d\geq
1$, there exist an $m\geq 1$ and a $G$-equivariant $(i_{1,m}\circ
\rho)_* ({\mathbf A})$-bundle
$$\U _m \to E\to E_{\Fa }G^{(d)}$$ which is (non-equivariantly) trivial
as an $\U _m$-principal bundle.
\end{theorem}

\begin{proof}
Let $Z$ be the mapping cylinder of the (unique up to homotopy)
map
$$B_{\Fa}(G,{\mathbf A})\to E_{\Fa }G$$ and let $B$ denote
$B_{\Fa}(G,{\mathbf A})$ in $Z$. Let $${\mathbf A}_m=(i_{1,m}\circ
\rho)_* ({\mathbf A})\text{\ \  and  \ \ }B_m=B_{\Fa}(G, {\mathbf
A}_m )$$ for $m\geq 1$, and let $$f:B\to B_{1}\text{,\ \ }
I_{k,m}:B_{k}\to B_{m}\text{,\  and \ \ } J_{k,m}:B_{k}\to B_{m}$$
be the maps induced respectively by $\rho $, $i_{k,m}$, and
$j_{k,m}$ for $1 \leq k< m$. For any $H\in \Fa $, the space
$B_{1}^H\simeq BC_{\U _1}(\rho\circ \alpha (H))$ is simply
connected. Therefore, we can extend $f$ to a $G$-map $f_2:Z^{(2)}\to
B_{1}$. Assume that we have a $G$-map
$$f_n:Z^{(n)}\to B_{k}$$ for $n\geq 2$ where $k\geq 1$. For $m>k$,
let the elements
$$\obs (I_{k,m}\circ f_n),\, \obs (J_{k,m}\circ f_n)\in
H_G^{n+1}(Z,B; \pi _{n}(B_{m}^{?}))$$ be the obstructions to
extending the restrictions $I_{k,m}\circ f_n|_{Z^{(n-1)}}$ and
$J_{k,m}\circ f_n|_{Z^{(n-1)}}$  to $G$-maps from $Z^{(n+1)}$ to
$B_{m}$ as in Proposition \ref{proposition:equivariant_obstruction}.
Since $\left\{(\U _k,i_k, j_k)\right\}_{k=1}^{\infty }$ is a
$|\Gamma |$-powering tower, we have $$\obs (I_{k,m}\circ
f_n)=|\Gamma |^{m-k} \obs (J_{k,m}\circ f_n)$$ for every $m >k$. By
Proposition \ref{proposition:spectral_sequence}, there is a
cohomology spectral sequence
$$E_2^{p,q}=\Ext^p_{\ZOr}(H_q(Z^?,B^?), \pi _{n}(B_{m}^{?}) )
\implies  H_G^{p+q}(Z,B; \pi _{n}(B_{m}^{?}))$$ and $|\Gamma |$
annihilates $H_q(Z^H,B^H)\iso \widetilde{H}_{q-1}(B^H)\iso
\widetilde{H}_{q-1}(BC_{\Gamma}(\alpha_H))$. Therefore, we can find
an $m>k$ such that
$$\obs (I_{k,m}\circ f_n)=0.$$
This implies that for any $d\geq 1$, there exists an $m>k$ and a
$G$-equivariant $(i_{1,m}\circ \rho)_* ({\mathbf A})$-bundle
$$\begin{CD}
\Gamma _m @>{}>> E @>{}>> E_{\Fa }G^{(d+1)}
\end{CD}$$ which is obtained as the pullback of the bundle
$E_{\Fa}(G,{\mathbf A}_m )\to B_m$ by the composition map
$$\begin{CD}
E_{\Fa }G^{(d+1)} @>>> Z^{(d+1)} @>{f_{d+1}}>> B_m.
\end{CD}$$
Since $E_{\Fa }G$ is contractible, we obtain a trivial $\U
_m$-principal bundle when we pullback this bundle to $E_{\Fa
}G^{(d)}$ by the inclusion map.
\end{proof}

\begin{corollary} \label{corollay:constructing smooth actions}
Let $G$ be a finite group and $M$ be a finite dimensional smooth
manifold with a smooth $G$-action. Let $\Fa$ denote the family of
isotropy subgroups of the $G$ action on $M$. Let $\G$ be a finite
group and
$$ {\mathbf A}=(\alpha _H)\in \lim _{\underset{H\in \Fa }{\longleftarrow
}} \Rep (H,\Gamma )$$ be a family of compatible representations.
Then, for every $\rho : \Gamma \to U(n)$, there exist a positive
integer $N$ and a smooth $G$-action on $M \times \bbS^N$ such that
for every $x \in M$, the $G_x$ action on the sphere $\{x\}\times
\bbS^N$ is given by the linear $G$-action on $\bbS(V^{\oplus k})$
where $V=\rho \circ \alpha_{G_x} $ and $k$ is some positive integer.
\end{corollary}

\begin{proof}
Let $\left\{(\U _k,i_k, j_k)\right\}_{k=1}^{\infty }$ be the
$|\Gamma |$-powering tower described in Example
\ref{example:Un_powering_tower}. Then, by Theorem
\ref{theorem:LuckOliverconstruction}, for any $d\geq 1$, there exist
an $m>1$ and a $G$-equivariant $(i_{1,m}\circ \rho)_* ({\mathbf
A})$-bundle
$$\U _m\to E\to E_{\Fa }G^{(d)}$$ which is trivial
as an $\U _m$-principal bundle. Consider the vector bundle
$$\bbC ^s\to  E\times
_{\U _m} \bbC ^s\overset{\pi}{\to} E_{\Fa }G^{(d)}$$ where
$s=n|\Gamma |^{m-1}$. Choose $d$ larger than the dimension of $M$.
We know that the isotropy subgroup of the $G$ action on $M$ are all
in $\Fa$. Hence we have a map $f: M\to E_{\Fa }G^{(d)}$ unique up to
homotopy and we can consider the following pullback
$$\xymatrix{
E \ar[d]^{p} \ar[r]^{}
& E\times
_{\U _m} \bbC ^s \ar[d]^{\pi}  \\
M  \ar[r]^{f\ \ } & E_{\Fa }G^{(d)}. }$$

Let $V$ be the direct sum of infinitely many copies of the regular
representation of $G$ over the real numbers $\bbR$. Let $BO(2s,V)$
denote the $G$-space of $2s$-planes in $V$ and $EO(2s,V)$ denote the
$G$-space whose points are pairs $(W,w)$ where $W$ is a $2s$-plane
in $V$ and $w \in W$. The map $EO(2s,V)\to BO(2s,V)$ defined by
$(W,w)\to W$ gives a $G$-equivariant vector bundle which is the
universal bundle of $2s$-dimensional $G$-equivariant vector bundles.
So we can consider $p$ as a pullback
$$\xymatrix{
E \ar[d]^{p} \ar[r]^{} & EO(2s,V) \ar[d]^{}  \\ M  \ar[r]^{h\ \ \ \
\ } & BO(2s,V) }$$ for some map $h: M \to BO(2s, V)$. In fact, since
$M$ is a finite dimensional manifold, the same is true if we replace
$V$ with a direct sum of $q$ copies of the regular representation
for a large $q$ (see Proposition III.9.3 in \cite{osborn}).

Note that $h$ is $G$-homotopic to a smooth $G$-map (see Theorem
VI.4.2 in \cite{bredon2}), so there is a smooth $G$-equivariant
vector bundle $p':E'\to M$ topologically equivalent to the
$G$-equivariant vector bundle $p:E\to M$. For every $x \in M$, the
$G_x$-action on $\bbS ((p')^{-1} (x))$ is the same as the
$G_x$-action on $\bbS (p^{-1} (x))$ which is given by the linear
$G_x$-action on $\bbS( (\rho \circ \alpha_{G_x})^{\oplus k})$ where
$k$ is some positive integer.

The bundle $p:E\to M$ has a (nonequivariant) topological
trivialization, so does $p': E' \to M$. Now a continuous
trivialization can be replaced by a smooth trivialization leading to
a diffeomorphism $\bbS(E') \diffeo M\times \bbS^N$ where $\bbS(E')$
is the total space of the corresponding sphere bundle and $N=2s-1$.
This is explained in detail in Chapter 4 of \cite{hirsch} (see also
Proposition 6.20 in \cite{lee}).  Note that the differential
structure on the product $M \times \bbS^N$ is the product
differential structure and $\bbS^N$ denotes the standard sphere, not
an exotic one.
\end{proof}

\section{Embedding fusion systems}
\label{section:ConstructionOfFiniteGroupGamma}

A key ingredient in the construction of an equivariant vector bundle
is the existence of a finite group $\G$ and a family of compatible
representations $ {\mathbf A}=(\alpha _H: H \to \Gamma)$. The
compatibility of representations $(\alpha_H)$ means that for each
map $c_g: H \to K$ induced by conjugation $c_g (h)=ghg^{-1}$, there
exists a $\gamma \in \Gamma $ such that the following diagram
commutes:
\begin{equation}\label{equation:compatible}
\xymatrix{H \ar[d]^{c_g} \ar[r]^{\alpha _H}
& \Gamma \ar[d]^{ c_{\gamma }}  \\
K  \ar[r]^{\alpha _K} & \Gamma}
\end{equation}
To find $\G$ and a family of compatible representations, we use an
intermediate finite group $S$ and define $\G$ in terms of $S$ and a
fusion system on $S$. More precisely, we assume that there is a
finite group $S$ and a family of maps $\iota _H: H \to S$ such that
the diagram (\ref{equation:compatible}) above comes from a diagram
of the following form:
\begin{equation}\label{equation:compatible extended}
\xymatrix{ H \ar[d]^{c_g} \ar[r]^{\iota _H} \ar@/^3pc/[rrr]^{\alpha
_H} & \iota _H (H) \ar@{>->}[d]^{f} \ar@{^{(}->}[r] & S
\ar@{^{(}->}[r]^{\iota }
& \Gamma \ar[d]^{ c_{\gamma }}  \\
K  \ar[r]^{\iota _K} \ar@/_3pc/[rrr]_{\alpha _K}
& \iota _K (K)  \ar@{^{(}->}[r]
& S  \ar@{^{(}->}[r]^{\iota }
& \Gamma }
\end{equation}

In general, the monomorphisms $f: \iota _H (H) \to \iota _K (K)$
that complete these diagram do not have to exist, but we assume that
they always exist. In fact, in our applications the maps $\iota_H$
are always injective, so we can take $f$ as the composition
$\iota_K\circ c_g \circ \iota _H ^{-1}$. Note that the monomorphisms
$f:\iota _H (H) \to \iota _K (K)$ do not only depend on the
conjugations $c_g$, but also depend on different choices of maps
$\iota_H$. These monomorphisms between subgroups of $S$ satisfy
certain properties and the best way to study them is via the theory
of abstract fusion systems. We now introduce the terminology of
fusion systems.

\begin{definition} Let $S$ be a finite group. A {\it fusion system} $\Fu$ on $S$
is a category whose objects are subgroups of $S$ and whose morphisms
are injective group homomorphisms where the composition of morphisms
in $\Fu $ is the usual composition of group homomorphisms and where
for every $P,Q \leq S$, the morphism set $\Hom _{\Fu} (P,Q)$
satisfies the following:
\begin{enumerate}
\item $\Hom _S(P,Q) \subseteq \Hom _{\Fu} (P,Q)$ where $\Hom _S
(P,Q)$ is the set of all conjugation homomorphisms induced by
elements in $S$.
\item For every morphism $\varphi$ in $\Hom _{\Fu} (P,Q)$, the
induced group isomorphism $P\to \varphi(P)$ and its inverse are also
morphisms in $\Fu$.
\end{enumerate}
\end{definition}

An obvious example of a fusion system is the fusion system $\Fu _S
(G)$ where $G$ is a finite group, $S$ a subgroup of $G$, and the set
of morphisms $\Hom _{\Fu} (P, Q)$ is defined as the set of all maps
induced by conjugations by elements of $G$. If $\Fu _1$ and $\Fu _2$
are two fusion systems on a group $S$, then we write $\Fu _1
\subseteq \Fu _2$ to mean that all morphisms in $\Fu _1$ are also
morphisms in  $\Fu _2$. We have the following:

\begin{lemma}\label{lemma:bridge} Let $G$ be a finite group and
$\cH$ be a family of subgroups of $G$. Let $S$ be a finite group and
$\{\iota_H : H \to S \ | \ H \in \cH \}$ be a family of maps.
Suppose that $\Fu$ is a fusion system on $S$ such that for every map
$c_g: H \to K$ induced by conjugation, there is a monomorphism $f$
in $\Fu$ such that the following diagram commutes
$$\xymatrix{
H \ar[d]^{c_g} \ar[r]^{\iota _H\ \ } & \iota _H (H) \ar@{>->}[d]^{f}
 \\
K  \ar[r]^{\iota _K\ \ } & \iota _K (K). }
$$
If $\G$ is a finite group which includes $S$ as a subgroup and
satisfies $\Fu \subseteq \Fu _S (\G)$, then the family of maps
$(\alpha _H)$, where $\alpha_H$ is defined as the composition
$$\alpha_H: H \maprt{\iota _H } S \hookrightarrow \G$$ for all $H \in \Fa$, is a compatible
family.
\end{lemma}

Given a fusion system on $S$, a good way to find a finite group
$\Gamma $ satisfying $\Fu \subseteq \Fu _S (\G)$ is to use certain
$S$-$S$-bisets. Before we explain this construction, we first
introduce some terminology about bisets.

An $S$-$S$-biset $\Omega $ is a non-empty set where $S$ acts both
from right and from left in such a way that for all $s, s' \in S$
and $x\in \Omega $, we have $(sx)s'=s(xs')$. Let $\Omega $ be an
$S$-$S$-biset, $Q$ be a subgroup of $S$, and $\varphi : Q \to S$ be
a monomorphism. Then, we write ${}_Q \Omega  $ to denote the
$Q$-$S$-biset obtained from $\Omega $ by restricting the left
$S$-action to $Q$ and we write ${}_{\varphi} \Omega  $ to denote the
$Q$-$S$-biset obtained from $\Omega $ where the left $Q$-action is
induced by $\varphi $.

We now discuss the construction of the finite group $\Gamma $ for a
given biset. This construction is the same as the construction given
by S. Park in \cite{park} for saturated fusion systems on
$p$-groups. Let $S$ be a finite group and $\Omega $ be an
$S$-$S$-biset. Let $\Gamma _{\Omega }$ denote the group of
automorphisms of the set $\Omega $ preserving the right $S$-action.
Define $\iota :S\to \Gamma _{\Omega }$ as the homomorphism
satisfying $\iota (s)(x)=sx$ for all $x\in \Omega $. If the left
$S$-action on ${\Omega }$ is free and $\Omega $ is non-empty, then
$\iota $ is a monomorphism, hence in that case we can consider $S$
as a subgroup of $\Gamma _{\Omega }$.

\begin{lemma}[Theorem 3, \cite{park}]\label{lemma:FstableBisets} Let $\Omega $
be an $S$-$S$-biset with a free left $S$-action and let $Q$ be a
subgroup of $S$ and $\varphi :Q\to S $ be a monomorphism. Then,
 $ _{\varphi }{\Omega } $ and $ _{Q}{\Omega } $ are
isomorphic as $Q$-$S$-bisets if and only if $\varphi $ is a morphism
in the fusion system $\Fu _S (\Gamma _{\Omega }).$
\end{lemma}

\begin{proof}Let $\eta:{} _{Q}{\Omega } \to {}_{\varphi }{\Omega } $
be a function. Note that $\eta$ is a $Q$-$S$-biset isomorphism if
and only if $\eta$ is an element in $\Gamma _{\Omega }$ and the
conjugation $c _{\eta }$ restricted to $Q$ is equal to $\varphi
:Q\to S$. This is because
$$c _{\eta }(q)(x)=\eta (q \eta ^{-1}(x))=\varphi (q)\eta (\eta ^{-1}(x))
=\varphi (q)(x)$$ for all $q\in Q$ and $x\in \Omega $.
\end{proof}

We make the following definition for the situation considered in
Lemma \ref{lemma:FstableBisets}.

\begin{definition}\label{definition:FStableBiset}
Let $\Fu$ be a fusion system on a finite group $S$. Then, a left
free $S$-$S$-biset ${\Omega }$ is called {\it left} $\Fu $-{\it
stable} if for every subgroup $Q \leq S$ and $\varphi \in \Hom_{\Fu}
(Q, S)$, the $Q$-$S$-bisets ${}_Q {\Omega } $ and $ {}_{\varphi}
{\Omega }$ are isomorphic.
\end{definition}

Hence, by Lemma \ref{lemma:FstableBisets}, we have the following.

\begin{theorem}\label{theorem:LeftFStable_FusionSystemInclusion}
Let $\Fu$ be a fusion system on a finite group $S$. If $\Omega $ is
a left $\Fu $-stable $S$-$S$-biset, then $\Fu\subseteq \Fu _S(\Gamma
_{\Omega })$.
\end{theorem}

This theorem together with Lemma \ref{lemma:bridge} gives an
explicit way to construct a finite group $\Gamma$ and a compatible
family of representations $(\alpha_H: H \to \Gamma)$. Note that if
$\Omega$ is also free as a right $S$-set, then the group $\Gamma$
can be described in a simple way as follows: If $|\Omega/S|=n$, then
$\Gamma$ is the wreath product $S \wr \Sigma _n :=(S\times \cdots
\times S)\rtimes \Sigma_n$ where the product of $S$'s is $n$-fold
and the symmetry group $\Sigma _n$ acts on the product by permuting
the coordinates. The fusion data is encoded in the way $S$ is
embedded in $\Gamma$. In general, the image of $\iota: S \to \G$ is
not in the product $S \times \cdots \times S$ (see \cite{park} for
more details).

For our constructions, we also need to find a representation of
$\Gamma $ such that its restriction via the maps $\alpha _H$ is in a
desired form. For this, we again use $S$ as an intermediate step,
start with a representation $V$ of $S$ and obtain a representation
of $\G$ in terms of $V$.

\begin{definition} Let $V$ be a left $\bbC S$-module and let $\Omega $ be a
$S$-$S$-biset. Then we define $\bbC \Gamma _{\Omega}$-module
$\widetilde V$ as the module
$$\widetilde{V}=\bbC \Omega \otimes _{\bbC S} V$$
where $\bbC \Omega $ is the permutation $\bbC S$-$\bbC S$-bimodule
with basis given by $\Omega$. The left $\bbC \Gamma
_{\Omega}$-action on $\bbC \Omega $ is given by evaluation of the
bijections in $\Gamma _{\Omega}$ at the elements of $\Omega $ and
 $\widetilde{V}$ is considered as a left $\bbC \Gamma
_{\Omega}$-module via this action.
\end{definition}

Note that every transitive $S$-$S$-biset is of the form $S\times _H
S$ for some $H \leq S \times S$, where $S\times _H S $ is the
equivalence class of pairs $(s_1, s_2)$ where $(s_1h_1, s_2)\sim
(s_1, h_2s_2)$ if and only if $(h_1, h_2)\in H$. The left and right
actions are given by usual left and right multiplication in $S$. An
$S$-$S$-biset is called bifree if both left and right $S$ actions
are free. It is clear from the above description that a transitive
bifree $S$-$S$-biset $S \times _H S$ has the property that $H \cap
(S \times 1)=1$ and $H \cap (1 \times S)=1$. Applying Goursat's
theorem, we obtain that $H$ is a graph of an injective map $\varphi
: Q \to S$ where $Q \leq S$. In this case we denote $H$ by
$$\Delta (\varphi )=\{ (s, \varphi(s) ) \ | \ s\in Q \}.$$ So, a bifree
$S$-$S$-biset is a disjoint union of bisets of the form
$S\times _{\Delta (\varphi)} S $ where $\varphi : Q \to S$ is a
monomorphism.

\begin{definition} Let ${\Omega }$ be a finite bifree $S$-$S$-biset.
Then we define the isotropy of ${\Omega }$ as the family
$$\Isot ({\Omega })=
\left\{ \varphi :Q\to S \left|\  S\times _{\Delta (\varphi)} S
\text{ is isomorphic to a transitive summand of }{\Omega } \right.
\right\}.$$
\end{definition}

It is known that every transitive biset can be written as a product
of five basic bisets (see Lemma 2.3.26 in \cite{bouc}). Since
$\Omega$ is bifree, only three of these basic bisets, namely
restriction, isogation, and induction, are needed to write the
transitive summands of $\Omega$ as a composition of basic bisets.
This gives us the following calculation:

\begin{proposition}\label{proposition:IsoOfConstructedRep}
Let $V$ be a left $\bbC S$-module and $\Omega $ be a bifree
$S$-$S$-biset. Let $\widetilde{V}$ be the $\bbC \Gamma
_{\Omega}$-module constructed as above. Then, for $H\leq S$, the
$\bbC H$-module $\Res ^{\Gamma _{\Omega}}_{H} \widetilde{V}$ is a
direct sum of modules in the form
$$\Ind ^{H}_{H\cap  Q^x} \Isog ^*(\varphi \circ c_{x}) \Res
^{S}_{\varphi (\leftexp{x}H\cap Q)} V$$ where $x\in S$ and
$\varphi:Q\to S$ is in $\Isot (\Omega )$.
\end{proposition}

\begin{proof} By writing the transitive summands of $\Omega$ as a
composition of the three basic bisets, we can express
$$\Res ^{\Gamma _{\Omega}}_{S} \widetilde{V}= \bbC  \Omega
\otimes_{\bbC S} V$$ as a direct sum of $\bbC S$-modules in the form
$$\Ind ^{S}_{Q} \Isog ^* (\varphi ) \Res ^{S}_{\varphi (Q)} V$$ where
$\varphi:Q\to S$ is in $\Isot (\Omega )$. Note that $\Isog ^*
(\varphi)$ is the contravariant isogation defined by $\Isog ^*
(\varphi) (M)=\varphi^* (M)$ where  $M$ is a $\varphi (Q)$-module.

Let $H$ be a subgroup of $S$. Then, the $\bbC H$-module $\Res
^{\Gamma _{\Omega}}_{H} \widetilde{V}$ is a direct sum of $\bbC
H$-modules in the form
$$\Res ^{S}_{H} \Ind ^{S}_{Q} \Isog ^* (\varphi )\Res ^{S}_{\varphi (Q)}
V.$$ Using the Mackey decomposition formula, we can decompose $\Res
^S _H \Ind ^S _Q$ further. We obtain  a direct sum with summands of
the form
$$\Ind ^{H}_{H\cap Q^x}\Isog^* (c_{x})\Res ^{Q}_{\leftexp{x}H\cap Q}
\Isog ^* (\varphi )\Res ^{S}_{\varphi (Q)} V$$ which is isomorphic
to
$$\Ind ^{H}_{H\cap  Q^x} \Isog ^*(\varphi \circ c_{x}) \Res
^{S}_{\varphi (\leftexp{x}H\cap Q)} V.$$ This completes the proof.
\end{proof}

This proposition shows that if we want to use this method of
construction of a finite group $\G$ using a left $\Fu$-stable biset
$\Omega$, we need to put some restrictions on the isotropy subgroups
of $\Omega$. The existence of left $\Fu$-stable bisets with certain
restrictions on their isotropy subgroups is an interesting problem
and we plan to discuss this in a future paper. For the main theorems
of this paper, it is possible to avoid this discussion by finding
specific bisets with desired properties using ad hoc methods. These
bisets will be described in the next section.

\section{Constructions of free actions on products of spheres}
\label{section:ConstructionsofFreeActions}

In this section, we prove our main theorems, Theorem
\ref{theorem:RankTwoPGroups} and
\ref{theorem:ElementaryAbelianIsotropy}, stated in the introduction.
We will first prove Theorem \ref{theorem:RankTwoPGroups} which
states that a $p$-group $G$ acts freely and smoothly on a product of
two spheres if and only if $\rk (G)\leq 2$. We start with a
well-known lemma which is often used as a starting point for
constructing free actions.

\begin{lemma}\label{lemma:bigcenter}
Let $G$ be a $p$-group with $\rk G=r$. If $\rk Z(G)=k$, then $G$
acts smoothly on a product of $k$-many spheres with isotropy
subgroups having rank at most $r-k$.
\end{lemma}

\begin{proof} Let the center of $G$ be of the form $Z(G)\cong \bbZ /p^{n_1}\times
\dots \times \bbZ / p^{n_k}$ with generators $a_1, \dots ,a_k$. For
$j\in \{1,2,\dots, k\}$, let $\chi _j:Z(G)\to\bbC $ denote the
one-dimensional representation of $Z(G)$ defined by $a_j\mapsto
e^{2\pi i/p^{n_j}}$, and $a_{j'}\mapsto 1$ for $j'\ne j$. Let
$\theta _j=\Ind ^G_{Z(G)}(\chi _j)$. Define $M=\bbS(\theta _1)\times
\cdots \times \bbS(\theta _k)$ with the diagonal $G$-action. Note
that $Z(G)$ acts freely on $M$, so if $H$ is an isotropy subgroup of
$G$, then we must have $H \cap Z(G)=\{1\}$. Thus, $H Z(G)\iso H
\times Z(G)$ is a subgroup of $G$. This proves that $k+ \rk H \leq
r$.
\end{proof}

The above lemma, in particular, says that if $\rk G=r$ and $\rk
Z(G)=r-1$, then $G$ acts smoothly on a product of $(r-1)$-many
spheres with rank one isotropy subgroups. When $p$ is odd, all rank
one $p$-groups are cyclic. In the case of $2$-groups, in addition to
cyclic groups, we also have the family of generalized quaternions
$Q_{2^n}$ where $n\geq 3$. In either case, given a finite collection
rank one $p$-groups, we can find a maximal rank one $p$-group into
which all other rank one $p$-groups can be embedded. In the proof of
Theorem \ref{theorem:RankTwoPGroups}, $S$ will be this maximal rank
one $p$-group. So, for $p$ is odd, $S$ will be a cyclic group of
order $p^N$ and for $p=2$, it will be a quaternion group $Q_{2^N}$
where $N$ is a large enough positive integer.

As a fusion system on $S$ we will always consider the fusion system
$\Fu$ where all the monomorphisms between subgroups of $S$ are in
$\Fu$. For this $S$ and $\cF$, we construct left $\Fu$-stable bisets
with reasonable isotropy structures. We construct these bisets using
a more general lemma. Before we state this lemma, we introduce some
definitions and notations.

\begin{definition} Let $\Fu $ be a fusion system on a finite group $S$.
Then we say $K$ is an $\Fu $-{\it characteristic} subgroup of $S$ if
for any subgroup $L\leq K$ and for any morphism $\varphi : L \to S $
in $\Fu$, there exists a morphism $\widetilde{\varphi }:K\to K$  in
$\Fu$ such that $\widetilde{\varphi }(l)=\varphi (l)$ for all $l \in
L$.
\end{definition}

We define $\Out _{\Fu} (S)$ as the quotient $\Aut _{\Fu} (S)/\Inn
(S)$. Notice that for $\varphi $ in $\Aut _{\Fu} (S)$, the
isomorphism class of the $S$-$S$-biset $S\times _{\Delta (\varphi)}
S $ depends only on the left coset $\varphi \Inn (S)$ in $\Out
_{\Fu} (S)$ since we can define an isomorphism from  $S\times
_{\Delta (\varphi)} S $ to $S\times _{\Delta (\varphi \circ c_x )} S
$ given by $[s_1,s_2]\mapsto [s_1x,s_2]$ for $x \in S$.

\begin{lemma} \label{lemma:Fchar_leftFbiset} Let $\Fu $ be a fusion system
on a finite group $S$ and $K$ be an $\Fu $-characteristic subgroup
of $S$. Assume that $\Omega $ is the $S$-$S$-biset defined as
follows
$$\Omega =\coprod _{\varphi \in \Out_{\Fu }(K)} S\times _{\Delta(\varphi )} S.$$
Then the $S$-$S$-biset $\Omega $ is left $\Fu $-stable.
\end{lemma}

\begin{proof}
First note that for any $\Fu $-morphism $\psi :K \to S$, the
$K$-$S$-bisets $_{K}\Omega $ and ${}_{\psi} \Omega $ are isomorphic.
We can prove this by considering the isomorphism from the
$K$-$S$-biset $S \times _{\Delta (\varphi )} S$ to the $K$-$S$-biset
${}_{\psi} (S \times _{\Delta (\varphi \circ \psi^{-1} )} S)$ given
by $[s_1,s_2]\mapsto [\psi (s_1), s_2]$. Here notice that $\psi :K
\to S$ implies $\psi (K)=K$ since $K$ is an $\Fu $-characteristic
subgroup of $S$. So, $\varphi \circ \psi^{-1}$ is also in $\Aut
_{\Fu} (K)$.

Now take any subgroup $Q \leq S$ and $\psi \in \Hom_{\Fu} (Q, S)$.
We want to show that $Q$-$S$-bisets ${}_Q {\Omega } $ and $
{}_{\psi} {\Omega }$ are isomorphic. We can think of a $Q$-$S$-biset
as a left $(Q\times S)$-set by defining the left $(Q\times
S)$-action by $(q, s)x=q x s^{-1}$ for all $q\in Q $ and $s \in S$.
This allows us to apply the usual theory of left sets to bisets. In
particular, to show that ${}_Q {\Omega } $ and $ {}_{\psi} {\Omega
}$ are isomorphic, it is enough to show that for every $H \leq
Q\times S$, the number of fixed points of left $H$- and
$H_{\psi}$-actions on $\Omega$ are equal where $H_{\psi}=\{(\psi
(x),y) \vv (x,y)\in H \}.$

Take any subgroup $H \leq Q\times S$. If $H$ is not a group in the
form $\Delta(\theta )$, where  $\theta :L\to S$ is a $\Fu $-morphism
and $L$ is a subgroup of $K$, then  $|\Omega ^H|=0=|\Omega
^{H_{\psi}}|$. If $H$ is a group in the form $\Delta(\theta)$ where
$\theta :L\to S$ is a $\Fu $-morphism and $L$ is a subgroup of $K$,
then there exists a morphism $\widetilde{\psi }:K\to K$ in $\Fu$
such that $\widetilde{\psi }(l)=\psi (l)$ for all $l\in L$. This
implies that $|\Omega ^H|=|\Omega ^{H_{\psi}}|$ because
$K$-$S$-bisets $_{K}\Omega $ and ${}_{\widetilde{\psi }} \Omega $
are isomorphic. From these we can conclude that $Q$-$S$-bisets ${}_Q
{\Omega } $ and $ {}_{\psi} {\Omega }$ are isomorphic.
\end{proof}

Lemma \ref{lemma:Fchar_leftFbiset} is used to show that the bisets
given in the following two examples are left $\Fu$-stable.

\begin{example}\label{example:rankonePgroup}
Let $p$ be a prime number and $S$ be the cyclic group of order $p^N$
where $N>1$. Let $\Fu $ be the fusion system on $S$ such that all
monomorphisms between the subgroups of $S$ are morphisms in $\Fu$.
Define
$$\Omega =\coprod _{\varphi \in \Aut(S)} S\times _{\Delta(\varphi )} S.$$
By Lemma \ref{lemma:Fchar_leftFbiset}, it is clear that $\Omega $ is
left $\Fu $-stable. So, by Theorem
\ref{theorem:LeftFStable_FusionSystemInclusion}, $\Fu \subseteq \Fu
_S (\Gamma _{\Omega})$. This implies that if a finite group $G$ acts
on a space with cyclic $p$-group isotropy, then its isotropy
subgroups can be embedded in $\G _{\Omega}$ in a compatible way. To
obtain the maps $\alpha_H : H \to \G _{\Omega}$, we first choose a
family of injective maps $\iota _H : H \to S$ for all isotropy
subgroups $H$, then we apply Lemma \ref{lemma:bridge} to conclude
that the compositions $\alpha_H=\iota \circ \iota _H$ form a
compatible family of maps.

As a representation $V$ of $S$, we can take the one dimensional
complex representation given by multiplication with the $p^N$-th
root of unity. Then, for every $H\leq S$, the representation $
\Res^{\G} _H \widetilde V$ is isomorphic to the direct sum
$\oplus_{\varphi} \Res^S_H \varphi^* (V)$, hence $H$ acts freely on
$\bbS (\widetilde V)$.\qed
\end{example}

\begin{remark} Note that the finite group $\Gamma _{\Omega}$ that is
constructed in the above example is exactly the same as the
construction given in Section 4.2 of \cite{unlu}. To see this, note
that the group $\Gamma _{\Omega }$ constructed above can be
expressed as wreath product $S \wr \Sigma _n $ where $n=|\Aut (S)|$.
We can write a specific group isomorphism $\Gamma _{\Omega} \to S
\wr \Sigma _n$ as follows: Observe that there is a $S$-$S$-biset
isomorphism between $S \times _{\Delta (\varphi)} S$ and the
$S$-$S$-biset ${}_{\varphi} S$ where the left $S$ action on
${}_{\varphi} S$ is via the automorphism $\varphi$. This isomorphism
is given by the map $\theta: S \times _{\Delta (\varphi)} S \to
{}_{\varphi} S$ defined by $\theta ([(s_1, s_2])=\varphi(s_1) s_2$.
So, we have
$$\Omega \iso \coprod _{\varphi \in \Aut(S)} {}_{\varphi} S.$$
Giving an ordering for the elements of $\Aut(S)$, we can write
$\Aut(S)=\{ \varphi_1,\dots , \varphi _n\}$. Now we define a map
from $\Gamma _{\Omega}$ to the wreath product $S \wr \Sigma
_n:=(S\times \cdots \times S) \rtimes \Sigma _n $ by sending an
automorphism $f : \Omega \to \Omega$ to the element $(f(e_1), \dots,
f(e_n) ; \sigma)$ where $e_i$ denotes the identity element of the
$i$-th component in the above disjoint union and $\sigma$ is the
permutation of the components induced by the automorphism $f$. This
map induces an isomorphism and under this isomorphism the embedding
$\iota : S \to \Gamma _{\Omega}$ becomes the embedding $S \to S \wr
\Sigma _n$ defined by $\iota (s)=(\varphi_1(s), \dots, \varphi _n
(s) ; \id)$. One can easily check that the representation of
$\Gamma_{\Omega}$ is also the same as the one given in Section 4.2
of \cite{unlu}. \qed
\end{remark}

\begin{example}\label{example:quaternionic}
Let $S$ be the generalized quaternion group $Q_{2^N}$ of order $2^N$
where $N\geq 3$. Let $\Fu $ be the fusion system on $S$ such that
all monomorphisms between the subgroups of $S$ are morphisms in
$\Fu$. Define
$$\Omega = S\times _{\Delta( \id _{C_2})} S$$
where $C_2$ is the unique cyclic group of order $2$ in $S$. Since
$C_2$ is a $\Fu$-characteristic subgroup and $\Aut
(C_2)=\{\id_{C_2}\}$, by Lemma \ref{lemma:Fchar_leftFbiset}, we can
conclude that $\Omega $ is left $\Fu $-stable, and hence $\Fu
\subseteq \Fu _S (\Gamma _{\Omega})$. Let $G$ be a finite group
acting on a space $X$ and let $\Fa$ denote the family of isotropy
subgroups of $G$-action on $X$. If every element in $\cH$ is a rank
one $2$-group, then we can choose a large $N$ and embed every
element $H \in \cH$ into $S=Q_{2^{N}}$ via some embedding $\iota _H
: H \to S$. Since the fusion system $\Fu$ includes all possible
monomorphisms between subgroups of $S$, the condition in Lemma
\ref{lemma:bridge} holds. So, the isotropy subgroups $H \in \cH$ can
be embedded in $\G _{\Omega}$ in a compatible way.

If $V$ is a representation of $S$, then for any $H \leq S$, the
representation $\Res ^S _H \widetilde V$ is isomorphic to a multiple
of the representation
$$\Ind _{C_2}^H \Res ^S _{C_2} V.$$ So, if we choose $V$ with the
property that $C_2$ acts freely on $\bbS(V)$, then $\Res ^S _H
\widetilde V$ also has the same property.\qed
\end{example}

The finite group $\Gamma_{\Omega}$ constructed in the above example
can also be expressed as a wreath product $S \wr \Sigma _n$ where
$n=|\Omega /S|=|S:C_2|$. But in this case, the image of $\iota : S
\to \Gamma _{\Omega}$ is not in the subgroup $S \times \cdots \times
S$. Under the natural projection $\pi: \G_{\Omega} \to \Sigma _n$,
the element $\pi (\iota (s))$, where $s\in S$, corresponds to the
permutation induced by the $s$ action on the coset set $S/C_2$.

Now we are ready to prove the following.

\begin{theorem}\label{theorem:rank1isotropy}
Let $G$ be a finite group acting smoothly on a manifold $M$ so that
the isotropy subgroup $G_x$ for every point  $x\in M$ is a rank one
$p$-group. Then, there exists a positive integer $N$ such that $G$
acts freely and smoothly on $M \times \bbS^N$.
\end{theorem}

\begin{proof} Let $\cH$ denote the family of isotropy subgroups of
$G$ action on $M$. By Examples \ref{example:rankonePgroup} and
\ref{example:quaternionic}, we know that there exists a finite group
$\Gamma$ and a family of compatible representations
$$ {\mathbf A}=(\alpha _H)\in \lim _{\underset{H\in \cH }{\longleftarrow
}} \Rep (H,\Gamma ).$$ In these examples we also showed that there
is a representation $\rho : \Gamma \to U(n)$ such that the
composition $\rho \circ \alpha _H : H \to U(n)$ is a free
representation for every $H\in \cH$.  Hence, by Corollary
\ref{corollay:constructing smooth actions}, $G$ acts freely and
smoothly on $M \times \bbS^N$ for some positive integer $N$.
\end{proof}

As an immediate corollary, we obtain the following.

\begin{theorem}\label{theorem:bigcenter}
Let $G$ be a $p$-group with $\rk G=r$. If $\rk Z(G) \geq r-1$, then
$G$ acts freely and smoothly on a product of $r$ spheres.
\end{theorem}

\begin{proof} This follows from Lemma \ref{lemma:bigcenter} and
Theorem \ref{theorem:rank1isotropy}.
\end{proof}

Now Theorem \ref{theorem:RankTwoPGroups} follows as a special case.

\begin{proof}[Proof of Theorem \ref{theorem:RankTwoPGroups}]
It is proved in \cite{heller} that $(\bbZ /p)^3 $ does not act
freely on a product of two spheres. Hence it is enough to construct
free actions of $p$-groups which has $\rk (G)\leq 2$. Note that
every finite $p$-group has a nontrivial center, so the existence of
such actions follows from Theorem \ref{theorem:bigcenter}.
\end{proof}

In the rest of the section, we prove Theorem
\ref{theorem:ElementaryAbelianIsotropy}. The proof is similar to the
above proof. We first consider the following example.

\begin{example}\label{example:elementaryabelian}
Let $p$ be a prime number and $S$ be the elementary abelian
$p$-group of order $p^N$ for some $N>1$. Let $\Fu $ be the fusion
system on $S$ such that all monomorphisms between the subgroups of
$S$ are morphisms in the fusion system $S$. Define
$$\Omega =\coprod _{\varphi \in \Aut(S)} S\times _{\Delta(\varphi)}
S.$$ Note that $\Omega $ is left $\Fu $-stable by Lemma
\ref{lemma:Fchar_leftFbiset} and hence, by Theorem
\ref{theorem:LeftFStable_FusionSystemInclusion}, we have $\Fu
\subseteq \Fu _{S} (\G)$. If $G$ is a finite group acting on a space
with elementary abelian isotropy subgroups, then we can find a
compatible family of representations $\alpha _H : H \to \G$ by first
choosing embeddings $\iota _H : H \to S$ and then by applying Lemma
\ref{lemma:bridge}.

In the following application, we can take the representation $V$ of
$S$ as the augmented regular representation $V=\bbC G-\bbC$, and
then construct $\widetilde{V}$ in the usual way. Note that for any
isotropy subgroup $H$, the representation $\alpha_H^* (\widetilde
V)$ is isomorphic to the direct sum
$$\bigoplus _{\varphi \in \Aut(S)}  (\iota _H )^* \varphi ^*(
V)$$ which is isomorphic to $(\iota _H)^* (V^{\oplus n})$ where
$n=|\Aut(S)|$. Note that we can choose $S$ so that when $H$ is an
isotropy subgroup of maximal rank the embedding $\iota _H: H \to S$
is an isomorphism. The action of an isotropy subgroup $H$  on
$\bbS(\widetilde V)$ will have no fixed points if $H$ has maximal
rank.
\end{example}

\begin{proof}[Proof of Theorem \ref{theorem:ElementaryAbelianIsotropy}]
Let $\cH$ denote the family of isotropy subgroups of $G$ action on
$M$. By Example \ref{example:elementaryabelian}, there is a finite
group $\Gamma$ and a family of compatible representations
$$ {\mathbf A}=(\alpha _H)\in \lim _{\underset{H\in \cH }{\longleftarrow
}} \Rep (H,\Gamma )$$ together with a representation $\rho : \Gamma
\to U(n)$ such that for every $H\in \cH$ of maximal rank, the
representation $\rho_H=\rho \circ \alpha_H :H \to U(n)$ has the
property that $H$ acts on $\bbS(\rho _H)$ without fixed points. By
Corollary \ref{corollay:constructing smooth actions}, there is a
smooth action on $M \times \bbS^{n_1}$ for some positive integer
$n_1$ such that for every $G_{x} \in \Fa$ of maximal rank, $G_x$
action on $\{ x\}\times \bbS^{n_1}$ is without fixed points. So, the
isotropy subgroups of $G$ action on $M \times \bbS^{n_1}$ has rank
$\leq k-1$. Repeating the argument recursively, we can conclude that
$G$ acts freely and smoothly on $M \times \bbS^{n_1} \times \cdots
\bbS^{n_k}$ for some positive integers $n_1,\dots, n_k$.
\end{proof}

The proof of Corollary \ref{corollary:ExtraSpecial} follows easily
from Theorem \ref{theorem:ElementaryAbelianIsotropy}. To see this,
observe that if $G$ is an (almost) extraspecial $p$-group of rank
$r$, then every subgroup which intersects trivially with the center
is an elementary abelian subgroup with rank less than or equal to
$r-1$. This is because the Frattini subgroup of $G$ is included in
the center $Z(G)$ of $G$ and that $Z(G)$ is cyclic. Let $a$ be a
central element of order $p$ in $G$. Let $\chi $ be the
one-dimensional representation of $\la a \ra$ defined by $a\mapsto
e^{2\pi i/p}$, and define $\theta =\Ind ^G_{\la a \ra}(\chi )$.
Then, $G$ action on $M=\bbS(\theta )$ has all its isotropy groups
elementary abelian with rank less than or equal to $r-1$. So, the
result follows from Theorem \ref{theorem:ElementaryAbelianIsotropy}.

Note that Theorem \ref{theorem:ElementaryAbelianIsotropy} applies to
a larger class of groups than (almost) extra-special $p$-groups. For
example, if $G$ is a $p$-group such that the elements of order $p$
in the Frattini subgroup $\Phi (G)$ of $G$ are all central, then the
action constructed in Lemma \ref{lemma:bigcenter} will satisfy the
assumptions of Theorem \ref{theorem:ElementaryAbelianIsotropy}, so
we can obtain free smooth actions of these groups on $r$ many
spheres where $r$ is the rank of the group. A particular example of
such a group would be a $p$-group $G$ which is a central extension
of two elementary abelian $p$-groups.

%\bibliographystyle{ih}
%\bibliography{ihmain}
%\end{document}
%%%%%%%%%%%%%%%%%%%%%%%%%%%%%%%%%%
\providecommand{\bysame}{\leavevmode\hbox
to3em{\hrulefill}\thinspace}
\providecommand{\MR}{\relax\ifhmode\unskip\space\fi MR }
\providecommand{\MRhref}[2]{%
  \href{http://www.ams.org/mathscinet-getitem?mr=#1}{#2}
} \providecommand{\href}[2]{#2}

\end{document}